\documentclass[a4paper,12pt]{article}
\usepackage[top=2.5cm,bottom=2.5cm,left=2.5cm,right=2.5cm]{geometry}
\usepackage{cite,amsmath}
    \usepackage[margin=1cm,%
                font=small,%
                format=hang,%
                labelsep=period,%
                labelfont=bf]{caption}
\usepackage{comment}

\usepackage[algoruled,linesnumbered]{algorithm2e}
\usepackage{algorithmic}

\usepackage{amsfonts,epsf,here,graphicx}
\usepackage{latexsym,multirow,rotating}
\usepackage{pgf,tikz,subfigure}
\usepackage{epstopdf}

\newtheorem{theorem}{\bf Theorem}[section]

\newtheorem{lemma}[theorem]{\bf Lemma}
\newtheorem{proposition}[theorem]{\bf Proposition}

\newtheorem{remark}[theorem]{\bf Remark}

\newcommand{\qed}{\hfill $\square$ \bigskip}
\newcommand{\s}{\color{black}}

\begin{document}

\baselineskip=0.30in
\vspace*{30mm}

\begin{center}
{\LARGE \bf Computing weighted Szeged and PI indices from quotient graphs}
\bigskip \bigskip

{\large \bf Niko Tratnik
}
\bigskip\bigskip

\baselineskip=0.20in

\textit{Faculty of Natural Sciences and Mathematics, University of Maribor, Slovenia} \\
{\tt niko.tratnik@um.si, niko.tratnik@gmail.com}

\bigskip\medskip

(Received June 2, 2019)

\end{center}

\noindent
\begin{center} {\bf Abstract} \end{center}
The weighted Szeged index and the weighted vertex-PI index of a  graph $G$ are defined as {\s $wSz(G) = \displaystyle \sum_{e=uv \in E(G)} (\deg (u) + \deg (v))n_u(e)n_v(e)$ and $wPI_v(G) = \displaystyle \sum_{e=uv \in E(G)} (\deg(u) + \deg(v))( n_u(e) + n_v(e))$}, respectively, {\s where $n_u(e)$ denotes the number of vertices closer to $u$ than to $v$ and $n_v(e)$ denotes the number of vertices closer to $v$ than to $u$}. Moreover, the weighted edge-Szeged index and the weighted PI index are defined analogously. As the main result of this paper, we prove that if $G$ is a connected graph, then all these indices can be computed in terms of the corresponding indices of weighted quotient graphs with respect to a partition of the edge set that is coarser than the $\Theta^*$-partition. If $G$ is a benzenoid system or a phenylene, then it is possible to choose a partition of the edge set in such a way that the quotient graphs are trees. As a consequence, it is shown that for a benzenoid system the mentioned indices can be computed in sub-linear time with respect to the number of vertices. Moreover, closed formulas for linear phenylenes are also deduced. However, our main theorem is proved in a more general form and therefore, we present how it can be used to compute some other topological indices.

\baselineskip=0.30in

\section{Introduction}
{\s Topological indices are numerical quantities of a graph that describe some of its structural properties. In mathematical chemistry, pharmacy, and in environmental sciences such indices are used for the QSAR/QSPR studies in which physicochemical properties of compounds are correlated with their molecular structure. Therefore, such indices are usually referred to as molecular descriptors \cite{dehmer2}. However, topological indices are becoming indispensable also in the theory of complex systems and networks \cite{estrada}. More precisely, they are applied for measuring many local and global properties in communications networks, neural networks, social
networks, biological networks, etc. As a consequence, topological indices can be essentially used in various scientific fields. However, molecular descriptors usually need to be calculated on a large family of molecules and hence, effective methods and algorithms for the calculation are needed.

In this paper, we develop methods for computing some modifications of two well known distance-based molecular descriptors, the Szeged index and the PI index, and show how these methods can be used on important chemical graphs. It is known that the mentioned indices have many applications in drug modelling \cite{drug}, in networks \cite{klavzar-2013,pisanski}, and correlate highly with physicochemical properties and biological activities of a large number of diversified and complex compounds \cite{pi}. However, the indices considered in the present paper do not rely only on the distances in graphs, but also on the vertex degrees, which makes them even more interesting for measuring properties of chemical and other systems.}

Research on distance-based topological indices began in $1947$, when H.\,Wiener used the distances in the molecular graphs of alkanes to calculate their boiling points.
As a consequence, the \textit{Wiener index} of a connected graph $G$ was defined as
$$W(G) = \sum_{\lbrace u, v \rbrace \subseteq V(G)} d_G(u,v)$$
and nowadays it represents one of the most popular molecular descriptors. Later, I.\,Gutman 
introduced the \textit{Szeged index} of a connected graph $G$ as
$$Sz(G) = \sum_{e=uv \in E(G)}n_u(e)n_v(e),$$
where $n_u(e)$ denotes the number of vertices of $G$ whose distance to $u$ is smaller than the distance to $v$ and $n_v(e)$ denotes the number of vertices of $G$ whose distance to $v$ is smaller than the distance to $u$. The main motivation for the introduction of this index came from the fact that $Sz(T)=W(T)$ for any tree $T$. In 2000,
a similar molecular descriptor, named as the \textit{PI index}, was defined with
$$PI(G) = \sum_{e=uv \in E(G)}\big(m_u(e) + m_v(e)\big),$$
where $m_u(e)$ denotes the number of edges of $G$ whose distance to $u$ is smaller than the distance to $v$ and $m_v(e)$ denotes the number of edges of $G$ whose distance to $v$ is smaller than the distance to $u$.
Finally, the \textit{edge-Szeged index} 
and the \textit{vertex-PI index} 
were also introduced:
$$Sz_e(G) = \sum_{e=uv \in E(G)}m_u(e) m_v(e), \qquad PI_v(G)=\sum_{e=uv \in E(G)}\big(n_u(e) + n_v(e)\big).$$
Szeged and PI indices were extensively studied in the literature and therefore, many mathematical results are known for these indices. For some relevant recent investigations see \cite{aroc,arock1,klavzar_li,wang}. {\s Moreover, in \cite{tratnik} the reader can find the references to the papers where the above indices were introduced.}

On the other hand, the \textit{degree distance} (also called the \textit{Schultz index}) was proposed 
with the following formula {\s (see \cite{brez-trat} for some relevant references)}:
$$DD(G) = \sum_{\lbrace u,v \rbrace \subseteq V(G)}\big(\deg (u) + \deg (v)\big)d_G(u,v).$$
Inspired by this extension of the Wiener index, Ili\'{c} and Milosavljevi\'{c} \cite{ilic} proposed  modifications of the Szeged index and the vertex-PI index. Therefore, they introduced the \textit{weighted Szeged index} and the \textit{weighted vertex-PI index}, which are defined as
\begin{eqnarray*}
wSz(G) &= &  \sum_{e=uv \in E(G)}\big(\deg(u) + \deg(v)\big)n_u(e)n_v(e),\\ wPI_v(G) &= &  \sum_{e=uv \in E(G)}\big(\deg(u) + \deg(v)\big)\big( n_u(e) + n_v(e)\big).
\end{eqnarray*}
For some research on these indices see \cite{bok,naga,patt1,patt2} and \cite{ma1,ma2,patt,you}, respectively. However, it seems natural to consider also the \textit{weighted edge-Szeged index} and the \textit{weighted PI index}:
\begin{eqnarray*}
wSz_e(G) &= &  \sum_{e=uv \in E(G)}\big(\deg(u) + \deg(v)\big)m_u(e)m_v(e),\\ wPI(G) &= &  \sum_{e=uv \in E(G)}\big(\deg(u) + \deg(v)\big)\big( m_u(e) +m_v(e)\big).
\end{eqnarray*}

As the main result of the paper, we prove that for any connected graph $G$ these indices can be calculated from the corresponding indices of weighted quotient graphs obtained by a partition of the edge set that is coarser than the $\Theta^*$-partition. In addition, our main result is proved in a more general form. More precisely, it can be used to calculate the Szeged index and the vertex-PI index of any graph with one vertex-weight and one edge-weight or to calculate the edge-Szeged index and the PI index of a graph with two edge-weights. Such a method for computing a topological index is usually called a cut method \cite{klavzar-2015}. Recently, similar methods were developed for other indices, in particular for the Wiener index \cite{nad_klav}, the revised (edge-)Szeged index \cite{li}, the degree distance \cite{brez-trat}, the Mostar index \cite{tratnik_mostar}, etc. 
Furthermore, our results greatly generalize several earlier results (for an example, see \cite{cre-trat,klavzar-2015,tratnik}). 

By using our main result, we also describe methods that enable us to efficiently calculate the weighted (edge-)Szeged index and the weighted (vertex-)PI index on benzenoid systems and phenylenes. Moreover, we calculate the mentioned indices for a fullerene patch and deduce closed formulas for linear phenylenes. Finally, it is shown how the results can be used to compute also some other indices that are defined in a similar way.

\section{Preliminaries}

Unless stated otherwise, the graphs considered in this paper are simple, {\s finite, and connected}. For a graph $G$, the set of all the vertices is denoted by $V(G)$ and the set of edges by $E(G)$. Moreover, we define $d_G(u,v)$ to be the usual shortest-path distance between vertices $u, v\in V(G)$. In addition, the distance between a vertex $u \in V(G)$ and an edge $e=xy \in E(G)$ is defined as
\begin{equation*}
{d}_G(u,e) = \min \lbrace d_G(u,x), d_G(u, y) \rbrace\,.
\end{equation*}
Furthermore, for any $u \in V(G)$ we define the \textit{degree} of $u$, denoted by $\deg(u)$, as the number of vertices that are adjacent to $u$.
\smallskip

\noindent
Let $G$ be a graph and $e=uv$ an edge of $G$. Throughout the paper we will use the following notation:{\s
\begin{eqnarray*}
N_u(e|G) & = & \lbrace x \in V(G) \ | \ d_G(u,x) < d_G(v,x) \rbrace,\\ 
N_v(e|G) & = & \lbrace x \in V(G) \ | \ d_G(v,x) < d_G(u,x) \rbrace,\\
M_u(e|G) & = & \lbrace f \in E(G) \ | \ d_G(u,f) < d_G(v,f) \rbrace,\\
M_v(e|G) & = & \lbrace f \in E(G) \ | \ d_G(v,f) < d_G(u,f) \rbrace. 
\end{eqnarray*}}
\noindent
If $G$ is a graph, {\s we say that functions $w, \lambda: V(G) \rightarrow \mathbb{R}_0^+$ are \textit{vertex-weights} and functions $w', \lambda': E(G) \rightarrow \mathbb{R}_0^+$ are \textit{edge-weights}}, where $\mathbb{R}_0^+ = [0,\infty)$. A graph $G$ together with some weights will be called a \textit{weighted graph}. 
\smallskip

\noindent
Let $G$ be a connected graph, $e=uv \in E(G)$, $w$ a vertex-weight, and $\lambda'$ an edge-weight. We set
\begin{align*}
n_u(e|(G,w)) &= \sum_{x \in N_u(e|G)} w(x), & \quad n_v(e|(G,w)) &= \sum_{x \in N_v(e|G)} w(x),\\
m_u(e|(G,\lambda')) &= \sum_{f \in M_u(e|G)} \lambda'(f), & \quad m_v(e|(G,\lambda')) &= \sum_{f \in M_v(e|G)} \lambda'(f).
\end{align*}
\noindent
If $M_u(e|G) = \emptyset$, we also set $m_u(e|(G,\lambda'))=0$.
\smallskip

\noindent
Let $G$ be a connected graph, $w$ a vertex-weight, and $w',\lambda'$ two edge-weights. We now define the Szeged index and the vertex-PI index of $(G,w,w')$, the edge-Szeged index and the PI index of $(G,\lambda',w')$, and the \textit{total-Szeged index} of $(G,w,\lambda',w')$ in the following way:
\begin{eqnarray*}
Sz(G,w,w') & = & \sum_{e=uv \in E(G)} w'(e)n_u(e|(G,w))n_v(e|(G,w)),\\
PI_v(G,w,w') & = & \sum_{e=uv \in E(G)} w'(e)\big(n_u(e|(G,w)) + n_v(e|(G,w)) \big),\\
Sz_e(G,\lambda',w') & = & \sum_{e=uv \in E(G)} w'(e)m_u(e|(G,\lambda'))m_v(e|(G,\lambda')),\\
PI(G,\lambda',w') & = & \sum_{e=uv \in E(G)} w'(e)\big( m_u(e|(G,\lambda')) + m_v(e|(G,\lambda')) \big),\\
Sz_t(G,w,\lambda',w') & = &  \sum_{e=uv \in E(G)} \big[ w'(e) \big( n_u(e|(G,w)) + m_u(e|(G,\lambda')) \big)  \cdot \\
& & \cdot \big( n_v(e|(G,w)) + m_v(e|(G,\lambda')) \big) \big].
\end{eqnarray*}


\noindent
Two edges $e_1 = u_1 v_1$ and $e_2 = u_2 v_2$ of graph $G$ are in relation $\Theta$, $e_1 \Theta e_2$, if
$$d(u_1,u_2) + d(v_1,v_2) \neq d(u_1,v_2) + d(u_2,v_1).$$
Recall that the mentioned relation is also known as Djokovi\' c-Winkler relation.
Moreover, relation $\Theta$ is reflexive and symmetric, but not always transitive. Therefore, its transitive closure (i.e.\ the smallest transitive relation containing $\Theta$) will be denoted by $\Theta^*$. 
\smallskip

\noindent
The {\em hypercube} $Q_n$ of dimension $n$ is defined in the following way: 
all vertices of $Q_n$ are binary strings of length $n$
and two vertices of $Q_n$ are adjacent if the corresponding strings differ in exactly one position. A subgraph $H$ of a graph $G$ is called an \textit{isometric subgraph} if for each $u,v \in V(H)$ it holds $d_H(u,v) = d_G(u,v)$. Any isometric subgraph of a hypercube is called a {\em partial cube}. {\s It is known that partial cubes form a large class of graphs with many applications (for an example, see \cite{brez-trat,cre-trat,klavzar-2015,nad_klav}). More precisely, various families of molecular graphs belong to partial cubes (benzenoid systems, trees, phenylenes, cyclic phenylenes, polyphenyls). For more information about partial cubes see \cite{klavzar-book}.}
\smallskip

\noindent
The subgraph of $G$ induced by $S \subseteq V(G)$ will be denoted by $\langle S \rangle$. Moreover, a subgraph $H$ of $G$ is called {\it convex} if for arbitrary vertices $u,v \in V(H)$ every shortest path between $u$ and $v$ in $G$ is also contained in $H$. The following theorem proved by Djokovi\' c and Winkler presents two fundamental characterizations of partial cubes.
\begin{theorem} \cite{klavzar-book} \label{th:partial-k} For a connected graph $G$, the following statements are equivalent:
\begin{itemize}
\item [(i)] $G$ is a partial cube.
\item [(ii)] $G$ is bipartite, and $\langle N_a(e|G) \rangle $ and $\langle N_b(e|G) \rangle$ are convex subgraphs of $G$ for all edges $e=ab \in E(G)$.
\item [(iii)] $G$ is bipartite and $\Theta = \Theta^*$.
\end{itemize}
\end{theorem}
In addition, we recall that if $G$ is a partial cube and $M$ a $\Theta$-class of $G$, then $G \setminus M$ has exactly two connected components, namely $\langle N_a(e|G) \rangle $ and {\s $\langle N_b(e|G) \rangle$}, where $e=ab \in M$ {\s (the details can be found in \cite{klavzar-book})}.
\smallskip

\noindent
Let $ \mathcal{E} = \lbrace M_1, \ldots, M_r \rbrace$ be the $\Theta^*$-partition of the set $E(G)$. Then we say that a partition $\mathcal{F}=\lbrace F_1, \ldots, F_k \rbrace$ of $E(G)$ is \textit{coarser} than $\mathcal{E}$
if each set $F_i$ is the union of one or more $\Theta^*$-classes of $G$. In such a case $\mathcal{F}$ will be shortly called a \textit{c-partition}.
\smallskip

\noindent
Suppose $G$ is a graph and $F \subseteq E(G)$. The \textit{quotient graph} $G / F$ is the graph whose vertices are connected components of the graph $G \setminus F$, such that two components $X$ and $Y$ are adjacent in $G / F$ if some vertex in $X$ is
adjacent to a vertex of $Y$ in $G$. If $E=XY$ is an edge in $G/F$, then we denote by $\widehat{E}$ the set of edges of $G$ that have one end vertex in $X$ and the other end vertex in $Y$, i.e. $\widehat{E}=  \lbrace xy \in E(G)\,|\,x \in V(X), y \in V(Y) \rbrace $. 
\smallskip

\noindent
Throughout the paper, let $\lbrace F_1, \ldots, F_k \rbrace$ be a c-partition of the set $E(G)$. Moreover, the quotient graph $G/F_i$ will be shortly denoted as $G_i$ for any $i \in \lbrace 1, \ldots, k \rbrace$. In addition, we define the function $\ell_i: V(G) \rightarrow V(G_i)$ as follows: for any $u \in V(G)$, let $\ell_i(u)$ be the connected component $U$ of the graph $G \setminus F_i$ such that $u \in V(U)$. The next lemma was obtained in \cite{nad_klav}, but the proof can be also found in \cite{tratnik_grao}.

\begin{lemma} \cite{nad_klav,tratnik_grao} \label{distance}
If $u,v \in V(G)$ are two vertices, then 
$$d_G(u,v) = \sum_{i=1}^k d_{G_i}(\ell_i(u),\ell_i(v)).$$
\end{lemma}

\noindent
The following lemma follows directly by Lemma \ref{distance}.
\begin{lemma} \cite{tratnik_mostar} \label{sosednji}
If $e =uv \in F_i$, where $i \in \lbrace 1, \ldots, k \rbrace$, then $U=\ell_i(u)$ and $V=\ell_i(v)$ are adjacent vertices in $G_i$.
\end{lemma}

\begin{proof}
Obviously, for any $j \in \lbrace 1, \ldots, k \rbrace$, $j \neq i$, it holds $\ell_j(u) = \ell_j(v)$ and therefore $d_{G_j}(\ell_j(u),\ell_j(v))=0$. By Lemma \ref{distance}, we now obtain
$d_{G_i}(\ell_i(u),\ell_i(v))=1$. \qed
\end{proof}
\section{The main result}


In this section, we prove that the (edge-)Szeged index and the (vertex-)PI index of a weighted graph can be computed from the corresponding weighted quotient graphs. 

Let $G$ be a connected graph and let $\lbrace F_1, \ldots, F_k \rbrace$ be a c-partition of the set $E(G)$. Suppose that $w: V(G) \rightarrow \mathbb{R}_0^+$ and $w',\lambda': E(G) \rightarrow \mathbb{R}_0^+$ are given weights. Then we define the corresponding weights $w_i, \lambda_i: V(G_i) \rightarrow \mathbb{R}_0^+$ and {\s $w_i',\lambda_i': E(G_i) \rightarrow \mathbb{R}_0^+$} on the quotient graph $G_i$, where $i \in \lbrace 1, \ldots, k \rbrace$, in the following way:

\begin{itemize}
\item {\s $w_i(X) = \displaystyle \sum_{x \in V(X)} w(x)$} for any $X \in V(G_i)$,
\item {\s $\lambda_i(X) = \displaystyle \sum_{e \in E(X)} \lambda'(e)$} for any $X \in V(G_i)$, 
\item {\s $w_i'(E) = \displaystyle \sum_{e \in \widehat{E}} w'(e)$} for any $E \in E(G_i)$,
\item {\s $\lambda_i'(E) = \displaystyle \sum_{e \in \widehat{E}} \lambda'(e)$} for any $E \in E(G_i)$.
\end{itemize}

We start with the following lemma from \cite{tratnik_mostar}, which will be needed in the proof of the main result.

\begin{lemma} \label{pomoc} \cite{tratnik_mostar} If $e=uv \in F_i$, where $i \in \lbrace 1, \ldots, k \rbrace$, $U=\ell_i(u)$, $V=\ell_i(v)$, and $E=UV \in E(G_i)$, then
$$n_u(e|(G,w)) = n_U(E|(G_i,w_i)),$$
$$n_v(e|(G,w)) = n_V(E|(G_i,w_i)).$$
\end{lemma}

\noindent
The next lemma shows that the set of edges lying closer to some end vertex of an edge $e=uv$ can be obtained from the corresponding vertices and edges of a quotient graph. The main idea of the proof can be found inside the proof of Theorem 3.1 in \cite{li}. However, for the sake of completeness we briefly describe the proof.
 
\begin{lemma} \label{pomoc1}
{\s Suppose $e=uv \in F_i$, where $i \in \lbrace 1, \ldots, k \rbrace$, $U=\ell_i(u)$, $V=\ell_i(v)$, and $E=UV \in E(G_i)$. Then}
$$M_u(e|G) = \left( \bigcup_{X \in N_{U}(E|G_i)} E(X) \right) \bigcup \left( \bigcup_{F \in M_{U}(E|G_i)} \widehat{F} \right),$$
$$M_v(e|G) = \left( \bigcup_{X \in N_{V}(E|G_i)} E(X) \right) \bigcup \left( \bigcup_{F \in M_{V}(E|G_i)} \widehat{F} \right).$$
\end{lemma}

\begin{proof}
Let $f=wz \in E(G)$ be an edge different from $e$. Suppose that $f \in F_j$ for some $j \in \lbrace 1,\ldots, k \rbrace$.
Consider the following two options.
\begin{itemize}
\item [(a)] {\s For $i=j$}: by Lemma \ref{sosednji}, $F=\ell_i(w)\ell_i(z)$ is an edge of $G_i$. By using Lemma \ref{distance}, we can prove in {\s a similar way} as in \cite{li} (see part (2), Case 1 in the proof of {\s Theorem 3.1 \cite{li}}) that it holds
$$d_G(u,f) - d_G(v,f) = d_{G_i}(U, F) - d_{G_i}(V,F).$$
We can see from the above equality that $d_G(u,f) < d_G(v,f)$ holds if and only if $d_{G_i}(U, F) < d_{G_i}(V,F)$. Therefore, $f$ belongs to $M_u(e|G)$ if and only if there exists $F \in M_U(E|G_i)$ such that $f$ belongs to $\widehat{F}$.

\item [(b)] {\s For $i \neq j$}: obviously, it holds $X=\ell_i(w)=\ell_i(z)$. By using Lemma \ref{distance}, we can prove in {\s a similar way} as in \cite{li} (see part (2), Case 2 in the proof of {\s Theorem 3.1 \cite{li}}) that it holds
$$d_G(u,f) - d_G(v,f) = d_{G_i}(U, X) - d_{G_i}(V,X).$$
We can see from the above equality that $d_G(u,f) < d_G(v,f)$ holds if and only if $d_{G_i}(U,X) < d_{G_i}(V,X)$. Therefore, $f$ belongs to $M_u(e|G)$ if and only if there exists $X \in N_U(E|G_i)$ such that $f$ belongs to $E(X)$.
\end{itemize}
Combining cases (a) and (b), one can deduce the following equality:
$$M_u(e|G) = \left( \bigcup_{X \in N_{U}(E|G_i)} E(X) \right) \bigcup \left( \bigcup_{F \in M_{U}(E|G_i)} \widehat{F} \right).$$
The other equality can be shown in the same way. \qed
\end{proof}

\noindent
The following lemma will be also needed to prove the main theorem of this paper.

\begin{lemma} \label{pomoc2} {\s Suppose $e=uv \in F_i$, where $i \in \lbrace 1, \ldots, k \rbrace$, $U=\ell_i(u)$, $V=\ell_i(v)$, and  $E=UV \in E(G_i)$. Then}
$$m_u(e|(G,\lambda')) = n_U(E|(G_i,{\lambda}_i)) + m_U(E|(G_i,\lambda_i')),$$
$$m_v(e|(G,\lambda')) = n_V(E|(G_i,{\lambda}_i)) + m_V(E|(G_i,\lambda_i')).$$
\end{lemma}

\begin{proof} By Lemma \ref{pomoc1} we obtain
\begin{eqnarray*}
m_u(e|(G,\lambda')) & = & \sum_{f \in M_u(e|G)} \lambda'(f) \\
& = & \sum_{X \in N_{U}(E|G_i)} \left( \sum_{f \in E(X)} \lambda'(f) \right) + \sum_{F \in M_{U}(E|G_i)} \left( \sum_{f \in \widehat{F}}\lambda'(f) \right) \\
& = & \sum_{X \in N_{U}(E|G_i)} \lambda_i(X) + \sum_{F \in M_{U}(E|G_i)} \lambda_i'(F)  \\
& = & n_U(E|(G_i,{\lambda}_i)) + m_U(E|(G_i,\lambda_i')),
\end{eqnarray*}
which shows the first equality. The remaining equality can be proved in the same way. \qed
\end{proof}

\noindent
The main result of this paper reads as follows.
\begin{theorem} \label{glavni} Let $G$ be a connected graph and $\lbrace F_1,\ldots, F_k \rbrace$ a c-partition of the set $E(G)$. If $w$ is a vertex-weight and $w',\lambda'$ are edge-weights, then
\begin{eqnarray*}
Sz(G,w,w')&=&  \sum_{i=1}^k Sz(G_i,w_i,w_i'),\\
 PI_v(G,w,w') &=&  \sum_{i=1}^k PI_v(G_i,w_i,w_i'), \\
Sz_e(G,\lambda',w') & =& \sum_{i=1}^k Sz_t(G_i,\lambda_i,\lambda_i',w_i'),\\
 PI(G,\lambda',w') &=& \sum_{i=1}^k \big( PI_v(G_i,\lambda_i,w_i') + PI(G_i,\lambda_i',w_i') \big).
\end{eqnarray*}
\end{theorem}

\begin{proof}
Since $E(G) = \displaystyle\bigcup_{i=1}^k F_i$ and for any $i \in \lbrace 1,\ldots, k \rbrace$ it holds $$F_i= \bigcup_{E \in E(G_i)} \widehat{E},$$
we obtain
\begin{eqnarray*}
Sz(G,w,w') & = & \sum_{e=uv \in E(G)} w'(e) n_u(e|(G,w)) n_v(e|(G,w)) \\
& = & \sum_{i=1}^k \left( \sum_{e =uv \in F_i} w'(e)n_u(e|(G,w)) n_v(e|(G,w)) \right) \\
  & = & \sum_{i=1}^k \left( \sum_{E=UV \in E(G_i)} \left[ \sum_{e =uv \in \widehat{E}} w'(e) n_u(e|(G,w))n_v(e|(G,w))  \right] \right).
\end{eqnarray*}
\noindent
By Lemma \ref{pomoc} one can calculate
\begin{eqnarray*}
Sz(G,w,w') & = & \sum_{i=1}^k \left( \sum_{E=UV \in E(G_i)} \left[ \sum_{e =uv \in \widehat{E}} w'(e) n_U(E|(G_i,w_i)) n_V(E|(G_i,w_i)) \right] \right) \\
& = & \sum_{i=1}^k \left( \sum_{E=UV \in E(G_i)} n_U(E|(G_i,w_i)) n_V(E|(G_i,w_i)) \left[ \sum_{e  \in \widehat{E}} w'(e)  \right] \right) \\
& = & \sum_{i=1}^k \left( \sum_{E=UV \in E(G_i)} w_i'(E)  n_U(E|(G_i,w_i))  n_V(E|(G_i,w_i))  \right) \\
& = & \sum_{i=1}^k Sz(G_i,w_i,w_i'),
\end{eqnarray*}
which completes the first part of the proof. The proof for the vertex-PI index is almost the same. Analogously, for the edge-Szeged index we compute
\begin{eqnarray*}
Sz_e(G,\lambda',w') & = & \sum_{e=uv \in E(G)} w'(e) m_u(e|(G,\lambda')) m_v(e|(G,\lambda')) \\
& = & \sum_{i=1}^k \left( \sum_{e =uv \in F_i} w'(e)m_u(e|(G,\lambda')) m_v(e|(G,\lambda')) \right) \\
  & = & \sum_{i=1}^k \left( \sum_{E=UV \in E(G_i)} \left[ \sum_{e =uv \in \widehat{E}} w'(e) m_u(e|(G,\lambda'))m_v(e|(G,\lambda'))  \right] \right).
\end{eqnarray*}
By using Lemma \ref{pomoc2} and similar calculations as above, we now get 
\begin{eqnarray*}
Sz_e(G,\lambda',w')= \sum_{i=1}^k Sz_t(G_i,\lambda_i,\lambda_i',w_i'),
\end{eqnarray*}
which finishes the proof for the edge-Szeged index. The calculation for the PI index can be done in a similar way, the only difference being that the sum should be partitioned into the two sums. \qed
\end{proof}

\noindent
The above result gives a method for computing the weighted (edge-)Szeged index and the weighted (vertex-)PI index of any connected graph.

\begin{theorem} \label{metoda_sz}
If $G$ is a connected graph and $\lbrace F_1,\ldots, F_k \rbrace$ is a c-partition of the set $E(G)$, then
\begin{align*}
wSz(G)&=  \sum_{i=1}^k Sz(G_i,w_i,w_i'), & wPI_v(G) &=  \sum_{i=1}^k PI_v(G_i,w_i,w_i'), \\
wSz_e(G) & = \sum_{i=1}^k Sz_t(G_i,\lambda_i,\lambda_i',w_i'), & wPI(G) &= \sum_{i=1}^k \big( PI_v(G_i,\lambda_i,w_i') + PI(G_i,\lambda_i',w_i') \big),
\end{align*}
where $w_i,\lambda_i: V(G_i) \rightarrow \mathbb{R}_0^+$ and $w_i',\lambda_i': E(G_i) \rightarrow \mathbb{R}_0^+$ are defined as follows: 
\begin{itemize}
\item $w_i(X)$ is the number of vertices in a connected component $X$ of $G \setminus F_i$,
\item $\lambda_i(X)$ is the number of edges in a connected component $X$ of $G \setminus F_i$,
\item {\s $w_i'(E) = \displaystyle \sum_{xy \in \widehat{E}}({\deg}(x) + {\deg}(y))$} for any $E \in E(G_i)$,
\item $\lambda_i'(E) = |\widehat{E} |$ for any $E \in E(G_i)$. With other words, if $E=XY$, then $\lambda_i'(E)$ is the number of edges between connected components $X$ and $Y$.
\end{itemize}
\end{theorem}

\begin{proof}
For a graph $G$, we introduce the weights $w: V(G) \rightarrow \mathbb{R}_0^+$, $w',\lambda': E(G) \rightarrow \mathbb{R}_0^+$ as follows: $w(x) = 1$ for any $x \in V(G)$ and $\lambda'(e)=1$ for each $e \in E(G)$. Moreover, for any $e=xy \in E(G)$ we set $w'(e)=\textrm{deg}(x) + \textrm{deg}(y)$. Obviously, it holds that $wSz(G)=Sz(G,w,w')$, $wPI_v(G)=PI_v(G,w,w')$, $wSz_e(G)=Sz_e(G,\lambda',w')$, and $wPI(G)=PI(G,\lambda',w')$. The result now follows by Theorem \ref{glavni}. \qed
\end{proof}
\begin{remark}
If $G$ is a bipartite graph, it obviously holds $n_u(e) + n_v(e) = |V(G)|$ for any edge $e=uv$ in $G$. Therefore, in such a case the weighted vertex-PI index can be also computed as {\s $wPI_v(G) = |V(G)| \displaystyle \sum_{e =uv \in E(G)} (deg(u) + deg(v)) = |V(G)| \cdot M_1(G)$}, where $M_1(G)$ denotes the well known first Zagreb index \cite{dos}.
\end{remark}

In the rest of the section, we apply Theorem \ref{metoda_sz} to a fullerene patch $G$, see Figure \ref{fullerene} ($G$ is actually a subgraph of the well known buckminsterfullerene). The same figure also shows the $\Theta^*$-classes of $G$, which are denoted by $M_1,M_2,M_3,M_4,M_5,M_6$. We can easily see that relation $\Theta$ is not transitive and therefore, $G$ is not a partial cube. Note that the revised edge-Szeged index and the Mostar index were computed for this graph in \cite{li} and  \cite{tratnik_mostar}, respectively.

\begin{figure}[h!] 
\begin{center}
\includegraphics[scale=0.7,trim=0cm 0cm 0cm 0cm]{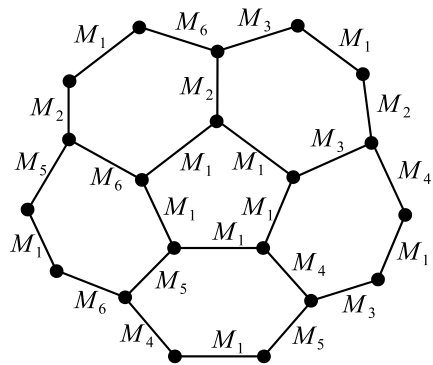}
\end{center}
\caption{\label{fullerene} Graph $G$ with its $\Theta^*$-classes \cite{tratnik_mostar}.}
\end{figure}

As in \cite{tratnik_mostar}, we define $F_1=M_1$ and $F_2 = M_2 \cup M_3 \cup M_4 \cup M_5 \cup M_6$. Obviously, $\lbrace F_1,F_2 \rbrace$ is a c-partition of $E(G)$. Next, the quotient graphs $G_1$ and $G_2$ together with the corresponding weights are depicted in Figure \ref{kvoc1} (the weights are calculated in terms of Theorem \ref{metoda_sz}). 
\begin{figure}[h!] 
\begin{center}
\includegraphics[scale=0.7,trim=0cm 0cm 0cm 0cm]{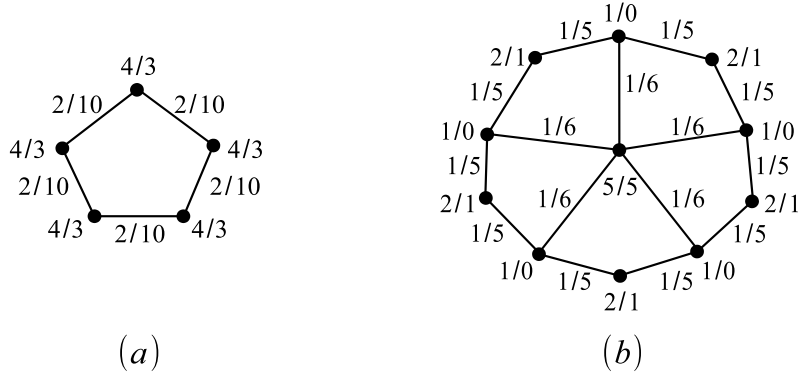}
\end{center}
\caption{\label{kvoc1} Quotient graphs (a) $G_1$ and (b) $G_2$ with vertex-weights $w_i/\lambda_i$ and edge-weights $\lambda_i'/w_i'$, $i \in \lbrace 1,2 \rbrace$.}
\end{figure}
Hence, by Theorem \ref{metoda_sz} the weighted Szeged index can be computed as follows:
\begin{eqnarray*}
wSz(G) & = & Sz(G_1,w_1,w_1') + Sz(G_2,w_2,w_2') \\
& = & \big( 5 \cdot 10 \cdot 8 \cdot 8 \big) + \big( 5 \cdot 6 \cdot 15 \cdot 5 + 10 \cdot 5 \cdot 15 \cdot 5 \big)   \\
& = & 3200 + 6000 = 9200.
\end{eqnarray*}
\noindent
Similar calculations give the other three indices:
\begin{eqnarray*}
wPI_v(G) & = & PI_v(G_1,w_1,w_1') + PI_v(G_2,w_2,w_2') \\
& = & \big( 5 \cdot 10 \cdot (8 + 8) \big) + \big(5 \cdot 6 \cdot (15 + 5) + 10 \cdot 5 \cdot (15 + 5)\big)   \\
& = & 800 + 1600 = 2400,\\
wSz_e(G) & = & Sz_t(G_1,\lambda_1,\lambda_1',w_1') + Sz_t(G_2,\lambda_2,\lambda_2',w_2') \\
& = & \big( 5 \cdot 10 \cdot (6+4) \cdot (6+4) \big) \\
&+ & \big( 5 \cdot 6 \cdot (8+10) \cdot (2+2) + 10 \cdot 5 \cdot (2+2) \cdot (8+10) \big) \\
& = & 5000 + 5760 = 10760,\\
wPI(G) & = & PI_v(G_1,\lambda_1,w_1') + PI(G_1,\lambda_1',w_1') + PI_v(G_2,\lambda_2,w_2') + PI(G_2,\lambda_2',w_2')\\
& = & \big(5 \cdot 10 \cdot (6+6) \big)  + \big( 5 \cdot 10 \cdot (4+4) \big) \\
& + & \big( 5 \cdot 6 \cdot (8+2) + 10 \cdot 5 \cdot (8+2) \big) + \big( 5 \cdot 6 \cdot (10+2) + 10 \cdot 5 \cdot (10+2) \big)\\
& = & 600 + 400 + 800 + 960 = 2760.
\end{eqnarray*}

\section{Benzenoid systems}

In this section, we apply Theorem \ref{metoda_sz} to benzenoid systems, which represent important molecular graphs \cite{gut_knjiga}. More precisely, the procedure for an efficient calculation of the weighted (edge-)Szeged index and the weighted (vertex-)PI index is described. It enables us to compute these indices in sub-linear time with respect to the number of vertices of a given benzenoid system. Note that analogous methods are already known for some other distance-based topological indices, see \cite{CK-1998,cre-trat1}.

\noindent
Let ${\cal H}$ be the hexagonal (graphite) lattice and let $Z$ be a cycle on it. A {\em benzenoid system} is the graph induced by vertices and edges of ${\cal H}$, lying on $Z$ and in its interior. For an example of a benzenoid system, see Figure \ref{ben_primer}. In addition, by $|Z|$ we denote the number of vertices in $Z$.

\begin{figure}[h!] 
\begin{center}
\includegraphics[scale=0.9,trim=0cm 0.4cm 0cm 0cm]{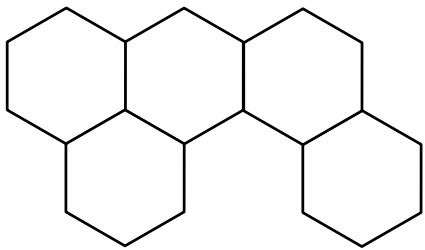}
\end{center}
\caption{\label{ben_primer} Benzenoid system $G$.}
\end{figure}

An \textit{elementary cut} of a benzenoid system $G$ is a line segment that starts at the center of a peripheral edge of a benzenoid system,
goes orthogonal to it and ends at the first next peripheral
edge of $G$. The main insight for our consideration
is that every $\Theta$-class of a benzenoid system
$G$ coincides with exactly one of its elementary cuts. Therefore, we can easily see by Theorem \ref{th:partial-k} that benzenoid systems are partial cubes \cite{klavzar-book}. 

The edge set of a benzenoid system $G$ can be naturally partitioned into sets $F_1, F_2$, and $F_3$ of edges of the same direction. Obviously, the partition $\lbrace F_1,F_2,F_3 \rbrace$ is a c-partition of the set $E(G)$. For $i \in \lbrace 1, 2, 3 \rbrace$, the quotient graph $G_i=G/F_i$ will be denoted as $T_i$. It is well known that $T_1$, $T_2$, and $T_3$ are trees \cite{chepoi-1996}. Moreover, we define weights $w_i$, $w_i'$, $\lambda_i$, $\lambda_i'$ on $T_i$, $i \in \lbrace 1,2,3 \rbrace$, as in Theorem \ref{metoda_sz}. If $G$ is a benzenoid system from Figure \ref{ben_primer}, then let $F_1$ be the set of all the vertical edges of $G$. Therefore, the edges from $F_1$ correspond to horizontal elementary cuts of $G$, see Figure \ref{quotient_tree} (a). The corresponding weighted quotient tree is shown in Figure \ref{quotient_tree} (b).

\begin{figure}[h!] 
\begin{center}
\includegraphics[scale=0.9,trim=0cm 0.4cm 0cm 0cm]{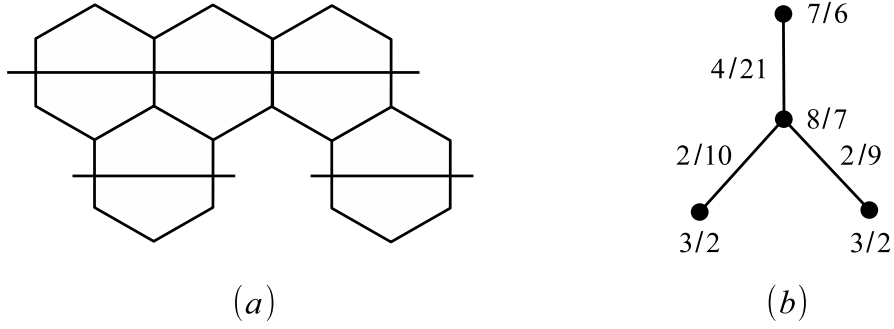}
\end{center}
\caption{\label{quotient_tree} (a) Graph $G$ with horizontal elementary cuts and (b) the corresponding quotient tree $T_1$ with vertex-weights $w_1/ \lambda_1$ and edge-weights $\lambda_1' / w_1'$.}
\end{figure}

\noindent
The next proposition is a direct consequence of Theorem \ref{metoda_sz}.
\begin{proposition} \label{metoda_ben}
If $G$ is a benzenoid system, then
\begin{align*}
wSz(G)&=  \sum_{i=1}^3 Sz(T_i,w_i,w_i'), & wPI_v(G) &=  \sum_{i=1}^3 PI_v(T_i,w_i,w_i'), \\
wSz_e(G) & = \sum_{i=1}^3 Sz_t(T_i,\lambda_i,\lambda_i',w_i'), & wPI(G) &= \sum_{i=1}^3 \big( PI_v(T_i,\lambda_i,w_i') + PI(T_i,\lambda_i',w_i') \big),
\end{align*}
where the trees $T_i$, $i \in \lbrace 1,2,3 \rbrace$, are defined as above and the weights are defined as in Theorem \ref{metoda_sz}.
\end{proposition}

\noindent
The following lemma will be also needed.
\begin{lemma} 
\label{lema_tree}
If $T$ is a tree with $n$ vertices, vertex-weights $w,\lambda$, and edge-weights $w',\lambda'$, then the indices $Sz(T,w,w')$, $PI_v(T,w,w')$, $Sz_t(T,\lambda,\lambda',w')$, and $PI(T,\lambda',w')$ can be computed in $O(n)$ time.
\end{lemma}

\begin{proof}
The proof is based on the standard BFS (breadth-first search) algorithm and it is similar to the proof of Proposition 4.4 in \cite{cre-trat} or to the proof of Lemma 4.1 in \cite{tratnik}. Therefore, we skip the details. \qed
\end{proof}

In \cite{chepoi-1996} it was shown that each weighted quotient tree $T_i$, $i \in \lbrace 1,2,3 \rbrace$, can be computed in linear time with respect to the number of vertices in a benzenoid system. Therefore, by Lemma \ref{lema_tree} and Proposition \ref{metoda_ben}, the weighted (edge-)Szeged and weighted (vertex-)PI indices can be computed in linear time. However, the mentioned indices can be computed even faster, i.e.\ in sub-linear time. To show this, we need the following lemma.

\begin{lemma} \label{tree_sub}
Let $G$ be a benzenoid system with a boundary cycle $Z$. If $i \in \lbrace 1,2,3 \rbrace$, then the tree $T_i$ and the corresponding weights $w_i$, $w_i'$, $\lambda_i$, $\lambda_i'$ can be computed in $O(|Z|)$ time.
\end{lemma}
\begin{proof}
The proof is similar to the proof of Lemma 3.1 in \cite{cre-trat1}. However, some additional insights are needed and therefore, we include the whole proof. The main idea is based on Chazelle's algorithm \cite{chazelle} for computing all vertex-edge visible pairs of edges of a simple (finite) polygon in linear time. Let ${\cal D}$ be the region in the plane bounded by cycle $Z$ such that $Z$ is included in $\cal{D}$. We define a \textit{cut segment} as a straight line segment lying completely in ${\cal D}$ and connecting two distinct vertices of $Z$. A \textit{cut segment of type $i$} is a cut segment perpendicular to the edges from the set $F_i$. We divide ${\cal D}$ into strips by all the cut segments of type $i$. Note that some strips can be triangles. Such a subdivision of ${\cal D}$ will be denoted by ${\cal D}_i$, see Figure \ref{tree2} (a). Moreover, any strip of ${\cal D}_i$ can take two values: $1$ and $\frac{1}{2}$. The strip takes value $1$ if it is a rectangle and value $\frac{1}{2}$ otherwise. By the algorithm of Chazelle, the subdivision ${\cal D}_i$ can be obtained in linear time. 
	

	Let ${\cal C}_i$ be the set of all cut segments of type $i$. Furthermore, let ${\cal C}_i'$ be the set of all vertices of $Z$ that are not on any cut segment of type $i$. Now we define a new graph $\Gamma_i$ whose vertices are the elements in the set ${\cal C}_i \cup {\cal C}_i'$ and two vertices of $\Gamma_i$ are adjacent if and only if the corresponding elements belong to a common strip of ${\cal D}_i$. From the definition of $\Gamma_i$ it follows that it is a tree. Note that $\Gamma_i$ can be derived from ${\cal D}_i$ in linear time. Moreover, an edge of $\Gamma_i$ is called {\em thick} if it is defined by a strip with value $1$ and {\em thin} otherwise. Every vertex from ${\cal C}_i$ is incident to exactly one thick edge, all remaining vertices of $\Gamma_i$ being incident only to thin edges, see Figure \ref{tree2} (b). We now introduce some weights on vertices and edges of the graph $\Gamma_i$, $i \in \lbrace 1,2,3 \rbrace$, see Figure \ref{tree2} (c).

\begin{itemize}
\item [$(i)$] \textit{Weight $a$ on the edges of $\Gamma_i$.} Let $f$ be an edge of $\Gamma_i$ and let ${\cal F}$ be the strip of ${\cal D}_i$ that corresponds to $f$. Edge $f$ gets a weight $a(f)$ that is equal to the number of edges in $G$ that lie completely in ${\cal F}$. Note that this weight can be computed by using lengths of corresponding cut segments.

\item [$(ii)$]  \textit{Weight $b$ on the vertices of $\Gamma_i$.} Any element of ${\cal C}_i'$ gets weight $1$. Moreover, let $x$ be an element of ${\cal C}_i$ and let $f$ be the thick edge incident to $x$. We now define $b(x)=a(f)$. Obviously, the weight of any vertex in ${\cal C}_i$ represents the number of vertices of $G$ lying on the corresponding cut segment.

\item [$(iii)$] \textit{Weight $c$ on the thick edges of $\Gamma_i$.} Let $f$ be a thick edge of $\Gamma_i$ and let ${\cal F}$ be the rectangular strip of ${\cal D}_i$ that corresponds to $f$. If $\widehat{\mathcal{F}}$ denotes the set of all the edges of $G$ that lie completely in ${\cal F}$, then we define {\s $c(f)= \displaystyle \sum_{e=uv \in \widehat{\mathcal{F}}} (\textrm{deg}(u) + \textrm{deg}(v))$}. Note that this can be computed in $O(|Z|)$ time since the vertices of degree 2 lie only on the boundary cycle $Z$ and all the other vertices of $G$ have degree 3.
\end{itemize}

\begin{figure}[h!] 
\begin{center}
\includegraphics[scale=0.9,trim=0cm 0.4cm 0cm 0cm]{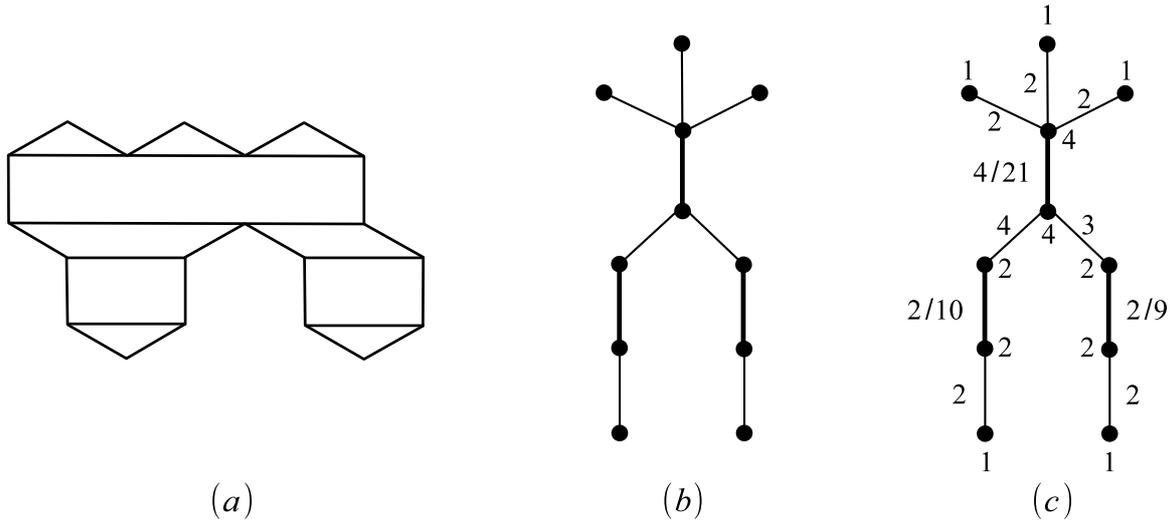}
\end{center}
\caption{\label{tree2} (a) Subdivision $\mathcal{D}_1$ for graph $G$, (b) the corresponding tree $\Gamma_1$, (c) the weight $b$ on vertices, the weight $a$ on thin edges, and weights $a/c$ on thick edges of $\Gamma_1$.}
\end{figure}

Finally, for any $i \in \lbrace 1,2,3 \rbrace$ we contract all the thin edges of $\Gamma_i$ to obtain the quotient tree $T_i$. However, the thick edges remain unchanged. Obviously, the weight $a$ on the thick edges of $\Gamma_i$ becomes the weight $\lambda_i'$. On the other hand, the weight $\lambda_i(u)$ of any vertex $u$ of tree $T_i$ can be obtained as the sum of all the weights $a(e)$ of thin edges $e$ from $\Gamma_i$ that are incident to the corresponding vertex $u$. Moreover, the weight $w_i(u)$ of a vertex $u$ from $T_i$ can be computed as the sum of all the weights $b(x)$  of vertices $x$ from $\Gamma_i$ that are identified with $u$.  Finally, the weight $c$ on the thick edges of $\Gamma_i$ corresponds to the weight $w_i'$. Therefore, the proof is complete. \qed
\end{proof}

\noindent
The following theorem is the main result of this section.

\begin{theorem}
If $G$ is a benzenoid system with a boundary cycle $Z$, then the indices $wSz(G)$, $wPI_v(G)$, $wSz_e(G)$, and $wPI(G)$ can be computed in $O(|Z|)$ time.
\end{theorem}

\begin{proof}
By Lemma \ref{tree_sub}, the quotient trees $T_i$, $i \in \lbrace 1,2,3 \rbrace$, and the corresponding weights $w_i$, $w_i'$, $\lambda_i$, $\lambda_i'$ can be computed in $O(|Z|)$ time. Moreover, by Lemma \ref{lema_tree}, the indices $Sz(T_i,w_i,w_i')$, $PI_v(T_i,w_i,w_i')$, $Sz_t(T_i,\lambda_i,\lambda_i',w_i')$, $PI_v(T_i,\lambda_i,w_i')$, and $PI(T_i,\lambda_i',w_i')$ can be computed in $O(|Z|)$ time. The result now follows by Proposition \ref{metoda_ben}. \qed
\end{proof}

\section{Phenylenes}

Beside benzenoid hydrocarbons, phenylenes represent another interesting class of polycyclic conjugated molecules, whose properties have been extensively studied, see \cite{zhu,zigert-2018} {\s and the references listed therein}. Therefore, in this section we describe a method for an efficient calculation of weighted Szeged and PI indices of phenylenes. Moreover, it is shown how our method can be used to easily obtain closed formulas for the mentioned indices.

A benzenoid system is said to be \textit{catacondensed} if all its vertices belong to the outer face. Let $G'$ be a catacondensed benzenoid system. If we add squares between all pairs of adjacent hexagons of $G'$, the obtained graph $G$ is called a \textit{phenylene}. We then say that $G'$ is the \textit{hexagonal squeeze} of $G$ and denote it by $HS(G)=G'$.  Moreover, two distinct inner faces of $G$ or $G'$ with a common edge are called \textit{adjacent}. The \textit{inner dual} of a benzenoid system $G'$  is the graph which has hexagons of $G'$ as vertices, two being adjacent whenever  the corresponding hexagons are also adjacent. Obviously, the inner dual of a catacondensed benzenoid system is a tree.

Let $G$ be a phenylene and $G'$ the hexagonal squeeze of $G$. The edge set of $G'$ can be naturally partitioned into sets $F_1'$, $F_2'$, and $F_3'$ of edges of the same direction. Denote the sets of edges of $G$ corresponding to the edges in $F_1'$, $F_2'$, and $F_3'$ by $F_1, F_2$, and $F_3$, respectively. Moreover, let $F_4 = E(G) \setminus (F_1 \cup F_2 \cup F_3)$ be the set of all the edges of $G$ that do not belong to $G'$. Again we can easily see that phenylenes are partial cubes and that the partition $\lbrace F_1,F_2,F_3,F_4 \rbrace$ is a c-partition of the edge set $E(G)$. For $i \in \lbrace 1, 2, 3, 4 \rbrace$, set $T_i = G /F_i$. As in the previous section, we can see that $T_1$, $T_2$, $T_3$, and $T_4$ are trees. In a similar way we can define the quotient trees $T_1', T_2', T_3'$ of the hexagonal squeeze $G'$. Then the tree $T_i'$ is isomorphic to $T_i$ for $i=1,2,3$ and $T_4$ is isomorphic to the inner dual of $G'$ \cite{zigert-2018}. The following proposition follows by Theorem \ref{metoda_sz}.

\begin{proposition} \label{metoda_ph}
If $G$ is a phenylene, then
\begin{align*}
wSz(G)&=  \sum_{i=1}^4 Sz(T_i,w_i,w_i'), & wPI_v(G) &=  \sum_{i=1}^4 PI_v(T_i,w_i,w_i'), \\
wSz_e(G) & = \sum_{i=1}^4 Sz_t(T_i,\lambda_i,\lambda_i',w_i'), & wPI(G) &= \sum_{i=1}^4 \big( PI_v(T_i,\lambda_i,w_i') + PI(T_i,\lambda_i',w_i') \big),
\end{align*}
where the trees $T_i$, $i \in \lbrace 1,2,3,4 \rbrace$, are defined as above and the weights are defined as in Theorem \ref{metoda_sz}.
\end{proposition}

\noindent
As in the previous section, it follows by Lemma \ref{lema_tree} and Proposition \ref{metoda_ph} that the indices $wSz(G)$, $wPI_v(G)$, $wSz_e(G)$, and $wPI(G)$ of a phenylene $G$ can be computed in linear time with respect to the number of vertices in $G$.

In the rest of the section, we apply Proposition \ref{metoda_ph} to calculate the closed formulas for an infinite family of phenylenes, i.e.\ for the linear phenylenes. A hexagon $h$ of a phenylene that is adjacent to exactly two squares is called \textit{linear} if the two vertices of degree 2 in $h$ are not adjacent. A phenylene is called \textit{linear} if every hexagon is adjacent to at most two squares and any hexagon adjacent to exactly two squares is a linear hexagon. The linear phenylene with exactly $n$ hexagons will be denoted by $PH_n$ ($n \geq 2$), see Figure \ref{linear_phe}.

\begin{figure}[h!] 
\begin{center}
\includegraphics[scale=0.9,trim=0cm 0.4cm 0cm 0cm]{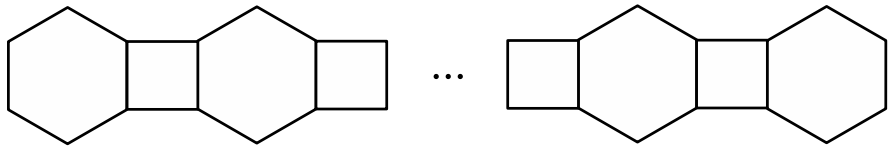}
\end{center}
\caption{\label{linear_phe} A linear phenylene.}
\end{figure}

\noindent
Let $F_1$ be the set of all the vertical edges in $PH_n$, let $F_4$ be the set of all the edges that do not lie on a hexagon, and let $F_2,F_3$ be the edges in the remaining two directions. We now compute the weighted quotient trees $T_1,T_2,T_3,T_4$, see Figure \ref{trees2_ph}.

\begin{figure}[h!] 
\begin{center}
\includegraphics[scale=0.65,trim=0cm 0.4cm 0cm 0cm]{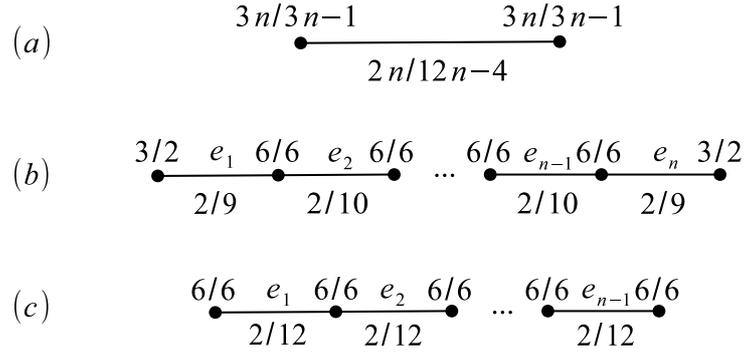}
\end{center}
\caption{\label{trees2_ph} Quotient trees (a) $T_1$, (b) $T_2 \cong T_3$, and (c) $T_4$ with vertex-weights $w_i/\lambda_i$ and edge-weights $\lambda_i'/w_i'$, $i \in \lbrace 1,2,3,4 \rbrace$.}
\end{figure}
\noindent
Firstly, we compute the corresponding indices of the weighted quotient tree $T_1$.
\begin{eqnarray*}
Sz(T_1,w_1,w_1') & = & (12n-4)\cdot (3n) \cdot (3n) = 108n^3 - 36n^2, \\
PI_v(T_1,w_1,w_1') & = & (12n-4)\cdot (3n + 3n) = 72n^2 - 24n,\\
Sz_t(T_1,\lambda_1,\lambda_1',w_1') & = & (12n-4)\cdot (3n-1) \cdot (3n-1) = 108n^3 - 108n^2 + 36n - 4,\\
PI_v(T_1,\lambda_1,w_1') & = & (12n-4)\cdot (3n-1 + 3n-1) = 72n^2 -48n + 8,\\
PI(T_1,\lambda_1',w_1') & = & (12n-4)\cdot (0 + 0) = 0.
\end{eqnarray*}
\noindent
Next, we calculate the corresponding indices of the weighted quotient tree $T_2$. Note that these indices are the same also for the quotient tree $T_3$.
\begin{eqnarray*}
Sz(T_2,w_2,w_2') & = & 2\cdot 9 \cdot 3 \cdot (6n-3) + \sum_{i=2}^{n-1} \big( 10 \cdot (6i-3) \cdot(6n-6i+3) \big)\\
& = & 60n^3 - 6n + 18, \\
PI_v(T_2,w_2,w_2') & = & 2\cdot 9 \cdot (3  + 6n-3) + \sum_{i=2}^{n-1} \big( 10 \cdot (6i-3 + 6n-6i+3) \big)\\
& = & 60n^2 - 12n, \\
Sz_t(T_2,\lambda_2,\lambda_2',w_2') & = & 2\cdot 9 \cdot (2 + 0) \cdot (6n-4 + 2n-2) \\
&+& \sum_{i=2}^{n-1} \big( 10 \cdot (6i-4 + 2i-2) \cdot(6n-6i+2 + 2n-2i) \big)\\
& = & \frac{1}{3} \big( 320n^3 - 480n^2 + 184n + 72 \big), \\
PI_v(T_2,\lambda_2,w_2') & = & 2\cdot 9 \cdot (2  + 6n-4) + \sum_{i=2}^{n-1} \big( 10 \cdot (6i-4 + 6n-6i+2) \big)\\
& = & 60n^2 - 32n + 4, \\
PI(T_2,\lambda_2',w_2') & = & 2\cdot 9 \cdot (0  + 2n-2) + \sum_{i=2}^{n-1} \big( 10 \cdot (2i-2 + 2n - 2i) \big)\\
& = & 20n^2 - 24n + 4. \\
\end{eqnarray*}
\noindent
Finally, we calculate these indices for the weighted quotient tree $T_4$.
\begin{eqnarray*}
Sz(T_4,w_4,w_4') & = & \sum_{i=1}^{n-1} \big( 12 \cdot (6i) \cdot(6n-6i) \big) = 72n^3 - 72n, \\
PI_v(T_4,w_4,w_4') & = & \sum_{i=1}^{n-1} \big( 12 \cdot (6i + 6n - 6i) \big) = 72n^2 - 72n,\\
Sz_t(T_4,\lambda_4,\lambda_4',w_4') & = & \sum_{i=1}^{n-1} \big( 12 \cdot (6i + 2i-2) \cdot(6n-6i + 2n-2i-2) \big)\\
& = & 128n^3 - 192n^2 + 112n - 48, \\
PI_v(T_4,\lambda_4,w_4') & = & \sum_{i=1}^{n-1} \big( 12 \cdot (6i + 6n - 6i) \big) = 72n^2 - 72n,\\
PI(T_4,\lambda_4',w_4') & = & \sum_{i=1}^{n-1} \big( 12 \cdot (2i -2 + 2n - 2i - 2) \big) = 24n^2 - 72n + 48.\\
\end{eqnarray*}
\noindent
Now we apply Proposition \ref{metoda_ph} to compute closed formulas for the weighted Szeged index, the weighted vertex-PI index, the weighted edge-Szeged index, and the weighted PI index of a linear phenylene $PH_n$, $n \geq 2$:
\begin{eqnarray*}
wSz(PH_n) & = & 300n^3 - 36n^2 - 84n +36,\\
wPI_v(PH_n) & = & 264n^2 - 120n,\\
wSz_e(PH_n) & = & \frac{1}{3} \big( 1348n^3 - 1860n^2 + 812n - 12 \big),\\
wPI(PH_n) & = & 328n^2 - 304n + 72.
\end{eqnarray*}
\section{Concluding remarks}

In the previous sections we designed new methods for the computation of weighted Szeged and PI indices. Moreover, it was shown how our methods can be used to efficiently calculate these indices on important molecular graphs. 

However, these indices can be also defined in a different way if we replace the term ${\deg}(u) + {\deg}(v)$ by the term ${\deg}(u) {\deg}(v)$, which is used in the definition of the Gutman index (see \cite{brez-trat} for more information about this index). Therefore, we obtain the following indices:
\begin{eqnarray*}
wSz^*(G)&=&  \sum_{e=uv \in E(G)}\deg(u)\deg(v)n_u(e)n_v(e), \\ wPI^*_v(G) &=&  \sum_{e=uv \in E(G)}\deg(u)\deg(v)\big( n_u(e) + n_v(e)\big), \\
wSz^*_e(G) & =& \sum_{e=uv \in E(G)}\deg(u)\deg(v)m_u(e)m_v(e), \\ wPI^*(G) &=&\sum_{e=uv \in E(G)}\deg(u)\deg(v)\big( m_u(e) + m_v(e)\big).
\end{eqnarray*}
 
\noindent 
The following theorem follows by Theorem \ref{glavni}.
\begin{theorem} \label{metoda_sz1}
If $G$ is a connected graph and $\lbrace F_1,\ldots, F_k \rbrace$ is a c-partition of the set $E(G)$, then
\begin{align*}
wSz^*(G)&=  \sum_{i=1}^k Sz(G_i,w_i,w_i'), & wPI^*_v(G) &=  \sum_{i=1}^k PI_v(G_i,w_i,w_i'), \\
wSz^*_e(G) & = \sum_{i=1}^k Sz_t(G_i,\lambda_i,\lambda_i',w_i'), & wPI^*(G) &= \sum_{i=1}^k \big( PI_v(G_i,\lambda_i,w_i') + PI(G_i,\lambda_i',w_i') \big),
\end{align*}
where $w_i,\lambda_i: V(G_i) \rightarrow \mathbb{R}_0^+$ and $\lambda_i': E(G_i) \rightarrow \mathbb{R}_0^+$ are defined as in Theorem \ref{metoda_sz}. Moreover, $w_i': E(G_i) \rightarrow \mathbb{R}_0^+$ is defined as follows: {\s $w_i'(E) = \displaystyle \sum_{xy \in \widehat{E}} \deg(x) \deg(y)$} for any $E \in E(G_i)$.
\end{theorem}

\noindent
It is easy to see that the results similar to the results from Section 4 and Section 5 hold also for these indices.

\section*{Funding information} 

\noindent The author Niko Tratnik acknowledge the financial support from the Slovenian Research Agency (research core funding No. P1-0297 and J1-9109). 

\baselineskip=17pt

\end{document}